\documentclass[11pt,reqno]{amsart}

\usepackage{setspace}
\usepackage{color}
\usepackage[all]{xy}

\usepackage{enumitem}

\calclayout

\usepackage{amsmath,amssymb}
\usepackage{amsthm}
\usepackage{mathabx}
\usepackage{tikz}
\usepackage{amsfonts}
\usepackage{pst-node}
\usepackage{tikz-cd}

\newtheorem{theorem}{Theorem}[section]
\newtheorem{proposition}[theorem]{Proposition}
\newtheorem{corollary}[theorem]{Corollary}
\newtheorem{lemma}[theorem]{Lemma}
\theoremstyle{definition}
\newtheorem{definition}[theorem]{Definition}
\newtheorem{example}[theorem]{Example}

\newtheorem{remark}[theorem]{Remark}
\newtheorem{question}[theorem]{Question}

\makeatletter
\@addtoreset{equation}{section}
\makeatother

\newcommand{\bZ}{{\mathbb Z}} 

\newcommand{\bF}{{\mathbb F}} 
\newcommand{\cA}{{\mathcal A}}

\newcommand{\cC}{{\mathcal C}}
\newcommand{\cD}{{\mathcal D}}

\newcommand{\cF}{{\mathcal F}}
\newcommand{\cG}{{\mathcal G}}

\newcommand{\cL}{{\mathcal L}}

\newcommand{\cT}{{\mathcal T}}

\providecommand{\AMS}{$\mathcal{A}$\kern-.1667em%
\lower.25em\hbox{$\mathcal{M}$}\kern-.125em$\mathcal{S}$}

\def\maprt#1{\smash{\,\mathop{\longrightarrow}\limits^{#1}\,}}

\newcommand{\Aut}{{\rm Aut}}
\newcommand{\Inn}{{\rm Inn}}
\newcommand{\Stab}{{\rm Stab}}
\renewcommand{\hom}{\mathrm{Hom}}
\newcommand{\mor}{\mathrm{Mor}}
\newcommand{\ext}{\mathrm{Ext}}

\newcommand{\res}{\mathrm{Res}}
\newcommand{\Ob}{\mathrm{Ob}}

\begin{document}
\title[Cohomology of infinite groups realizing fusion systems]{Cohomology of infinite groups realizing fusion systems } 

\author{Muhammed Sa\.id G\"undo\u gan}
\author{Erg\"un Yal\c cIn}
\address{Department of Mathematics, Bilkent University, 06800 
Bilkent, Ankara, Turkey}

\email{mgundogan@bilkent.edu.tr, yalcine@fen.bilkent.edu.tr}

\begin{abstract}  Given a fusion system $\mathcal{F}$ defined on a $p$-group $S$,  there exist infinite group models, 
constructed by Leary and Stancu, and Robinson, that realize $\cF$.  We study these models when $\cF$ is a fusion system 
of a finite group $G$ and prove a theorem which relates the cohomology of an infinite group model $\pi$ to the cohomology 
of the group $G$. We show that for the groups $GL(n,2)$, where $n\geq 5$, the cohomology of the infinite group obtained 
using the Robinson model is different than the cohomology of  the fusion system. We also discuss the signalizer functors 
$P\to \Theta(P)$ for infinite group models and obtain a long exact sequence for calculating the cohomology of a centric 
linking system with twisted coefficients.
\end{abstract}


\maketitle


\section{Introduction}\label{sec:Intro}  

Let $\Gamma$ be a discrete group. If $\Gamma$ has a finite $p$-subgroup $S$ such that every $p$-subgroup of $\Gamma$ is conjugate to a subgroup
of $S$, then $S$ is called a \emph{Sylow $p$-subgroup} of $\Gamma$.  The fusion system $\mathcal{F}_S(\Gamma)$ is defined as the category whose objects are subgroups of $S$, and whose morphisms are given by maps induced by conjugation by an element in $G$.  In general, a fusion system $\cF$ is a category whose objects are the subgroups of a finite $p$-group $S$ and whose morphisms are injective group homomorphisms satisfying certain properties. We say a fusion system $\cF$ is realized by a discrete group $\Gamma$, if $\Gamma$ has 
a Sylow $p$-subgroup $S$ such that $\cF=\cF_S(\Gamma)$. 

When a fusion system satisfies some further axioms that mimic Sylow theorems, it is called a \emph{saturated fusion system}.
Fusion systems realized by finite groups are the main examples of saturated fusion systems. There are exotic saturated fusion systems 
that are not realized by a finite group. However it has been shown independently by Leary and Stancu \cite{LearyStancu} and Robinson \cite{Robinson}  
that given a saturated fusion system $\cF$, there is always 
a discrete group $\pi$ that realizes $\cF$. These infinite group models are constructed as fundamental groups of certain graphs of groups. 
 
In this paper, we consider these infinite group constructions for a fusion system $\cF$ which is already realized
by a finite group $G$. We find these infinite group models interesting from the point of view of group cohomology and cohomology of categories, 
even in the case where $\cF$ is realized by a finite group.
The main aim of the paper is to prove a theorem which relates the cohomology 
of the fusion system $\cF$ of a group $G$ to the cohomology of an infinite group model $\pi$ that realizes $\cF$ and to provide an infinite family 
of examples where these two cohomology groups are not isomorphic.

Let $G$ be a finite group and $S$ be a Sylow $p$-subgroup of $G$. If there is subgroup $H \leq G$ which includes $S$ as Sylow $p$-subgroup such that $\cF_S(G)=\cF_S(H)$, then we say $H$ \emph{controls $p$-fusion in $G$}.  We say $G$ is \emph{$p$-minimal} if it has no proper subgroup
that controls $p$-fusion in $G$. Assume that $G$ is $p$-minimal and let $\pi$ be an infinite group realizing the fusion system $\cF_S(G)$
obtained by either the Leary-Stancu model or the Robinson model (using subgroups of $G$ as vertex groups, as described in Remark \ref{rem:RobinsonModel}).  Then we show that there is a surjective group homomorphism $\chi : \pi \to G$ whose kernel is a free group $F$ on which $G$ acts by conjugation (see Section \ref{sect:Storing} for details). 

The homomorphism $\chi :\pi \to G $ satisfies some extra properties that makes it a \emph{storing homomorphism} (see Definition \ref{def:Storing}).  In the case where there is a storing homomorphism $\chi: \pi \to G$ which takes a Sylow $p$-subgroup of $\pi$ to a Sylow $p$-subgroup of $G$, one can relate the mod-$p$ cohomology of $\pi$ to the mod-$p$ cohomology of $G$ via a direct sum decomposition (see Theorem \ref{thm:Storing}). As a consequence we obtain the following.

\begin{theorem}\label{thm:thmA}
Let $\cF =\cF_S(G) $ be a fusion system of a finite group $G$. Assume that $G$ is $p$-minimal, and let $\pi$ denote the infinite group realizing $\cF$ obtained by either the Leary-Stancu model or the Robinson model (as in Remark \ref{rem:RobinsonModel}).
Then there is a group extension $1 \to F \to \pi \to G \to 1$ where $F$ is a free group, and there is  an isomorphism of cohomology groups
$$H^{*-1} (G; \hom(F_{ab} , \bF_p ) )\oplus H^* (G; \bF_p) \cong H^* (\pi; \bF_p)$$ 
where $F_{ab}:=F/[F.F]$ denotes the abelianization of $F$.
\end{theorem}

In \cite[Thm 1.1]{LibmanSeeliger}, Libman and Seeliger  considers the cohomology of an infinite model group $\pi$ when $\cF$ is any saturated fusion system. In this case there is a map $f: B\pi \to |\cL|_p ^\wedge$ from the classifying space of $\pi$ to the $p$-completion of the associated centric linking system $\cL$.
They showed that the $\bF_p$-algebra homomorphism $H^* ( |\cL |; \bF_p ) \to H^* (\pi; \bF_p)$ induced by the map $f$ splits, and the splitting is given by the restriction map $\res^{\pi} _S : H^*(\pi ; \bF_p) \to H^* (S; \bF_p)$. This gives an isomorphism 
$$\ker (\res ^\pi _S ) \oplus H^* (|\cL |; \bF_p) \cong H^* (\pi; \bF_p)$$ (see Section \ref{subsect:Realizing} for details). As a consequence of Theorem \ref{thm:thmA} we obtain that if $\cF$ is a fusion system realized by a finite group $G$ that is $p$-minimal,  then the kernel of the restriction map  $\res ^\pi _S$ is isomorphic to $H^{*-1} (G; \hom (F_{ab}; \bF_p ))$   

In Section \ref{sect:Examples} 
we give some examples of storing homomorphisms and show that the mod-$p$ cohomology of an infinite group 
constructed using the Leary-Stancu model is not in general isomorphic to the cohomology of the fusion system that it realizes. 
For the Robinson model, we show that the groups $G=GL(n,2)$ for $n\geq 5$ give infinitely many such examples.
 
\begin{theorem}\label{thm:thmB}
Let $G=GL(n,2)$ for $n\geq 5$, and let $S$ be the Sylow-$2$ subgroup consisting of upper triangular matrices in $G$. 
Suppose that $\pi$ is the infinite group realizing $\mathcal{F}_S(G)$ constructed using the  Robinson model. Then 
$H^2(G , \bF_2) \not \cong H^2(\pi , \bF_2)$.     
\end{theorem}
 
In Section \ref{sect:GroupActions}, we consider finite group actions on graphs and show that under certain conditions 
group actions on graphs can be used to obtain infinite group models realizing fusion systems. For a finite group $G$ 
with $p$-rank equal to 2, we introduce a new infinite group
model whose vertex groups are normalizers of elementary abelian $p$-subgroups (see Theorem \ref{thm:Rank2}).

In Section \ref{sect:Signalizer}, we discuss signalizer functors $P \to \Theta (P)$ for infinite group models (see Definition \ref{def:Signalizer} for a definition of signalizer functor).
In the case where $\cF$ is the $p$-fusion system of a finite group $G$, we calculate the signalizer functors in terms of normalizers in the kernel of the storing homomorphism $\chi : \pi \to G$. 
For arbitrary fusion systems we show that for every $\cF$-centric $P$, the mod-$p$ homology of the group $\Theta (P)$ is zero in dimensions 
greater than $1$ (see Proposition \ref{pro:Homology}). In dimension 1 the homology group functor $P\to H_1 (\Theta (P); \bF_p)$ defines an $\bF_p \cL$-module. We denote this module by $H_1 \Theta$.  As a consequence of the vanishing of homology groups of $\Theta(P)$ at dimensions greater than 1, 
we obtain the following theorem.

\begin{theorem}\label{thm:thmC} Let $\cT:=\cT_S^c (\pi)$ denote the transporter category for an infinite group model $\pi$ defined on the $\cF$-centric subgroups of $S$, and let $\cL$ be the associated linking system defined by a signalizer functor $P \to \Theta (P)$. Then for every $\bF_p \cL$-module $M$, there is a long exact sequence $$ \cdots \to H^{n-1} (\cT; q^*M ) \to  \ext ^{n-2}_{R\cL} ( H_1\Theta , M ) \to  H^n ( \cL ; M) \to H^n (\cT ; q^* M) \to $$ $$\ext ^{n-1}_{R\cL} ( H_1 \Theta , M ) \to H^{n+1} (\cL ; M) \to \cdots$$
where $q^* M$ denote the $\bF_p \cT$-module obtained from $M$  via the quotient functor $q:\cT \to \cL$.
\end{theorem}


The paper is organized as follows: In Section \ref{sect:Prelim} we give necessary definitions and preliminary 
results that we use in the rest of the paper. In Section \ref{sect:Storing} we consider graphs of groups with a storing 
homomorphism and prove Theorem \ref{thm:thmA}. In Section \ref{sect:Examples} we consider some examples 
of storing homomorphisms and prove Theorem \ref{thm:thmB}. In Section \ref{sect:GroupActions}, we consider group actions on 
graphs and show that in some cases a group action on a graph can induce an infinite group model realizing the fusion system 
of the group. In Section \ref{sect:Signalizer} we discuss signalizer functors for infinite group models and prove 
Theorem \ref{thm:thmC}.

\vskip 5pt

\noindent 
{\bf Acknowledgement:}  The second author is supported by a T\" ubitak 1001 project (grant no. 116F194).


\section{Definitions and preliminary results}\label{sect:Prelim}

In this section we introduce necessary definitions and preliminary results for the rest of the paper. The readers familiar with
fusion systems and graphs of groups can skip most of this section. The standard reference for definitions on fusion systems 
is \cite{Craven} and for graphs of groups is \cite{Serre}. 

\subsection{Fusion systems}\label{subsect:Fusion}
 
Let $S$ be a finite $p$-group. A fusion system $\cF$ on $S$ is a category whose objects are 
the subgroups $P \leq S$, and for every $P,Q \leq S$ the morphism set $\hom_\cF (P,Q)$ consists of injective 
group homomorphisms $P\to Q$ with following properties:
 
\begin{itemize}
\item[(i)] For all $P,Q\leq S$ we have $\hom_S(P,Q)\subseteq \hom _{\cF} (P,Q)$, where $\hom_S(P,Q)$ is the set of all homomorphisms $c_s: P\to Q$ induced by conjugation with elements $s \in S$.
\item[(ii)]For any morphism  $\phi \in \hom _{\cF} (P, Q)$, the isomorphism $\phi:P \rightarrow \phi(P)$, and its inverse $\phi^{-1} : \phi(P) \to P$, are morphisms in $\cF$. 
\end{itemize}

Let $\Gamma$ be a discrete (possibly infinite) group, and let $S$ be a finite $p$-subgroup of $\Gamma$. We say $S$ is a \emph{Sylow $p$-subgroup} of $\Gamma$
if for every $p$-subgroup $P \leq \Gamma$ there is a $g\in \Gamma$ such that $gPg^{-1} \leq S$. When $S$ is a Sylow $p$-subgroup of
a discrete group $\Gamma$, the fusion system $\cF_S(\Gamma)$ is defined as the fusion system on $S$ whose morphisms $P\to Q$ are defined as the set of all maps $c_g: P \to Q$ induced by conjugations by elements in $\Gamma$.   If a fusion system satisfies some additional axioms it is called a \emph{saturated} fusion system (see \cite[Def 1.37]{Craven} for a definition). If $G$ is finite group with a Sylow $p$-subgroup $S$, then the fusion system $\cF_S(G)$ is saturated. 

Let $\cF$ be a fusion system, and let $P$ and $Q$ be two subgroups in $S$. If there is an isomorphism $f: P \to Q$ in $\cF$, then we say 
$P$ and $Q$ are \emph{$\cF$-conjugate} and denote this by $P \sim _{\cF} Q$. A subgroup $Q \leq S$ is called \emph{fully $\cF$-normalized} if $|N_S(Q)| \geq |N_S(R)|$ for every $R\leq S$ with $Q \sim _{\cF} R$. A subgroup $Q \leq S$ is called \emph{fully $\cF$-centralized} if $|C_S(Q)| \geq |C_S(R)|$ for every $R\leq S$ with $Q \sim _{\cF} R$. We say $Q$ is \emph{$\cF$-centric} if $C_S(R) \leq R$ for every $R \sim _{\cF} Q$. A subgroup $Q\leq S$ is called \emph{$\cF$-radical} if $O_p(\Aut _{\cF} (Q)) =\Inn (Q) $, where $O_p(G)$ denotes the largest normal $p$-subgroup in a group $G$.

If $\cF=\cF_S(G)$ for a finite group $G$ with a Sylow $p$-subgroup $S$, then $P\leq S$ is fully normalized in $\cF$ if and only if $N_S(P)$ is a Sylow $p$-subgroup of $N_G(P)$ (see \cite[Prop 1.38]{Craven}). A subgroup $P \leq S$ is $\cF$-centric if and only if $Z(P)$ is the Sylow $p$-subgroup of $C_G(P)$ (see \cite[Prop 4.43]{Craven}). In this case we have $C_G(P)=Z(P)\times C_G'(P)$ where $C'_G(P):=O_{p'} (C_G(P))$ denotes the largest normal subgroup of $C_G(P)$ whose order is coprime to $p$. A $p$-subgroup $P$ in $G$ is called \emph{$p$-centric} if it satisfies this property.

We say a $p$-subgroup $P$ of $G$ is \emph{$p$-radical} if $O_p(N_G(P)/P)=1$.  Note that a subgroup $P\leq S$ is $\cF_S(G)$-radical if and only if $O_p(N_G(P)/PC_G(P))=1$. Hence in general being $p$-radical and $\cF_S(G)$-radical are different conditions. However, the following holds.

\begin{lemma}\label{lem:CentricRadical} Let $G$ be finite group with a Sylow $p$-subgroup $S$, and let $P$ be a subgroup of $S$. If $P$ is $\cF_S(G)$-centric and $\cF_S(G)$-radical, then $P$ is $p$-centric and $p$-radical. In general the converse does not hold.
\end{lemma} 

\begin{proof} We have seen above that $P$ is $\cF_S(G)$-centric if and only if $P$ is $p$-centric. Assume that $P$ is not $p$-radical. Then there is a $p$-subgroup $Q$ of $N_G(P)$ such that $P \lhd Q \lhd N_G(P)$. Since $P$ is $p$-centric, $C_G(P)=Z(P)\times C_G'(P)$, hence  $PC_G(P) =PC_G'(P) $. Since $C'_G(P)$ has order coprime to $p$, we have $Q\cap C_G'(P)=1$. This gives that $QC_G' (P) /PC'_G(P)$ is a nontrivial normal $p$-subgroup in $N_G(P)/PC'_G(P)$. Hence $P$ is not $\cF_S(G)$-radical. 

To see that the converse that does not hold, let $G=D_{24}$ be the dihedral group of order 24 and let $P=O_2(G)$ be the normal cyclic subgroup of order $4$. The centralizer of $P$ in $G$ is the cyclic subgroup of order 12, so $P$ is $2$-centric. Since $N_G(P)/P\cong S_3$, $P$ is $2$-radical. However $P$ is not $\cF_S(G)$-radical since $N_G(P)/ PC_G(P)\cong C_2$ (see \cite[pg 11]{AKO}).
\end{proof}

Given a $p$-group $S$, the largest fusion system on $S$ is the system where morphisms from $P$ to $Q$ are all injective
homomorphisms $f: P \to Q$. This fusion system is denoted by $\cF^{\max} _S$. In general $\cF_S^{\max}$ is not a saturated fusion system.
The fusion system \emph{generated} by a collection of morphisms $\{ f_i : P_i \to Q_i \}$ is defined as the smallest subfusion system of $\cF_S ^{\max}$ that includes all the morphisms $f_i$. We denote this fusion system by $\langle f_i \, | \, i=1,\dots, n \rangle$. 

Alperin's theorem for fusion systems states that 
if $\cF$ is a saturated fusion system, then $\cF$ is generated by $\cF$-automorphisms of fully normalized, $\cF$-radical, $\cF$-centric subgroups of $S$ (see \cite[Thm 4.52]{Craven}).  


\subsection{Graphs of groups}\label{subsect:GraphGroups} 

A graph $\Gamma$ consists of two sets $E(\Gamma)$ and  $V(\Gamma)$, called the edges and vertices 
of $\Gamma$, an involution on $E(\Gamma)$ which sends $e$ to $\bar{e}$ where $e\neq \bar{e}$, and two maps 
$o, t :E (\Gamma) \rightarrow V(\Gamma)$ which satisfy $t(e)=o( \bar{e})$. Each edge $e$ is considered as an oriented edge, 
with origin $o(e)$ and terminus $t(e)$. The pair $\{ e, \bar{e}\}$ is called an unoriented edge. 

\begin{definition}
A graph of groups $(\mathcal{G}, Y)$ consists of a connected  nonempty graph $Y$ together with a function 
$\cG$ assigning \\ (i) to each vertex $v$ of $Y$ a group $G_v$ and 
to each edge $e$ of $Y$ a group $G_e$, such that $G_{\bar{e}}=G_e$ for all $e$,   and
\\ (ii) to each edge $e$, a monomorphism 
$\phi_e:G_e\rightarrow G_{t(e)}$. 
\end{definition}

The fundamental group of a graph of groups is a group that can be described by giving a presentation. Let 
$E$ denote the free group with a basis given by the edges of $Y$. For this presentation we denote the edges of $Y$ by $y$ 
and write $a^y$ for the image of $a \in G_y$ under the monomorphism $\phi _y$. Let $F(\cG, Y)$ denote the quotient group of the free 
product $$E\ast (\bigast _{v\in V(Y)} G_v)$$ by the normal subgroup $N$, where $N$ is the normal closure of the relations 
$$y a ^y y^{-1}=a^{\bar{y}} \ \text{  and   }  \bar{y}=y^{-1}$$
for all $y\in E(Y)$ and $a\in G_y$. Let $T$ be a maximal tree in $Y$, then we define the group $\pi (\cG, Y, T)$ to be the quotient group of $F(\cG, Y)$
subject to the relations $y=1$ if $y\in E(T)$. It can be shown that the isomorphism class of $\pi(\cG, Y, T)$ does 
not depend on the maximal tree $T$ that is chosen (see \cite[Proposition 20]{Serre}). We call the group 
$\pi(\cG, Y, T)$, \emph{the fundamental group of $(\cG, Y)$}, and denote it by $\pi (\cG, Y)$.

There is also a topological description of the fundamental group of a graph of groups as the fundamental group of 
a topological space. For a discrete group $G$, let $BG$ denote the classifying space of $G$.  For each edge $e$, 
there is a continuous map $B\phi_e : BG_e \to BG_{t(e) }$ induced by the group homomorphism $\phi_e : G_e \to G_{t(e)}$. 

\begin{definition} The \emph{total space $X(\cG, Y)$} of the graph of groups $(\cG, Y)$ is defined as the quotient space of 
$$ \Bigl (  \coprod _{v\in V(Y)} BG_v \Bigr ) \coprod \Bigl (  \coprod _{e\in E(Y)} (BG_e \times [0,1])  \Bigr )$$ 
by the identifications
$$BG_e\times[0,1]\rightarrow BG_{\bar{e}}\times[0,1] \ \text{  by }  (x,t)\mapsto(x,1-t)$$
and
$$BG_e\times\{1\}\rightarrow BG_{t(e)}\    \text{ by }  (x,1)\mapsto (B\phi_e)(x).$$
\end{definition}

Using van Kampen's theorem and some other arguments, it is possible to show that the fundamental group of $X(\cG, Y)$ is isomorphic
to the group $\pi (\cG, Y)$ defined above (see \cite[Prop 23, pg 204]{Cohen}). The space $X(\cG, Y)$ has a contractible universal covering, 
so it is a  classifying space for the group $\pi (\cG, Y)$ (see \cite[Thm 22]{LearyStancu}).

\begin{example} Two well-known examples of graph of groups are amalgamations and HNN-extensions. If $Y$ is a graph with 
one unoriented edge and two distinct vertices, and if $A$ and $B$ are the vertex groups and $C$ is the edge group with 
two monomorphism $A\hookleftarrow C\hookrightarrow B$, denoted by $i_A$ and $i_B$, then the fundamental group
$\pi$ is the amalgamated product $$A*_C B :=\langle A * B\, |\, i_A (c) i_B (c)^{-1}, \text{ for all } c \in C \rangle.$$  
For the HNN-extension, we take the graph $Y$ as a graph with one unoriented edge and one vertex, i.e. the graph is just a loop. 
If the vertex group is $A$, the edge group is $C$, and its monomorphism is the identity 
embedding $C\hookrightarrow A$ and $\phi:C\hookrightarrow A$, we obtain the HNN-extension 
$$A*_C := \langle A,t\, |\, tct^{-1}=\phi(c), \text{  for all }  c\in C \rangle.$$ 
When $Y$ is a finite graph  as it is assumed to be throughout this paper, one can express the fundamental
group $\pi$ as a finite sequence of amalgamations and HNN-extensions. Because of this it is important to understand these two examples
of graph of groups.  \qed
\end{example}

One of the key properties of graphs of groups is that vertex groups $G_v$ embed into the fundamental group of graphs of groups.

\begin{lemma}[Lemma 19, pg. 200, \cite{Cohen}] Let $(\cG, Y)$ be a graph of groups, let $Z$ be a connected subgraph of $Y$. 
Then the natural homomorphism $\pi (\cG|_Z, Z) \to \pi (\cG, Y)$ is a monomorphism. In particular, for any vertex $v$ of $Y$, the natural
homomorphism $i_v: G_v \to \pi (\cG, Y)$ is a monomorphism.  
\end{lemma}
 
Note that the natural map $i_v$ is the map induced by inclusion of $G_v$ into the free product $E\ast (\bigast _{v\in V(Y)} G_v)$.  Using the topological description of the fundamental group, it can also be described as the map induced by the inclusion of $BG_v$ into the total space $X (\cG, Y)$.


\subsection{Groups acting on graphs}\label{subsect:GroupActions}

Let $G$ be a group acting on a graph $X$. We say $G$ acts \emph{without inversion} if $ge\neq \bar{e}$ for every edge $e$ in $X$ and every $g\in G$.  Sometimes this type of action is called a cellular action.  Throughout the paper we will assume that all actions on graphs are without inversion. 
Assume that $X$ is a connected graph. Then we can define a graph of groups $(\cG, Y)$ on the graph $Y=X/G$ using the $G$-action on $X$. 
The vertex groups of $(\cG, Y)$ are the stabilizers of vertices of $X$ under $G$-action. The details of the construction of this graph of groups can be found in \cite[Section 5.4]{Serre}. 
The first structure theorem of the Bass-Serre theory is the following:

\begin{theorem}[Theorem 12, pg. 52, \cite{Serre}]\label{thm:BassSerre1} 
If $\pi$ is the fundamental group of a graph of groups $(\cG , Y)$, then there is a tree $T$ on which $\pi$ acts without inversions
such that the graph of groups associated to the  $\pi$ action on $T$ is isomorphic to $(\cG, Y)$.
\end{theorem}

The tree $T$ is usually called the universal cover of the graph of groups $(\cG, Y)$, and its construction is described in \cite[pg. 51]{Serre}. 

The second structure theorem of the Bass-Serre theory is in some sense a converse to Theorem \ref{thm:BassSerre1}.  
Let $G$ be a group acting on a graph $X$ without inversions, and let $(\cG , Y)$ be the associated graph of groups where $Y=X/G$. If  $\pi=\pi(\cG, Y)$, then there is a group homomorphism $\varphi : \pi \to G$ that takes the elements in $i(G_v)$ to the corresponding  stabilizer subgroups in $G$ and takes the HNN-extension generators $t$ to the corresponding group elements in $G$. There is also a map of graphs $\psi : T \to X$ from the universal cover $T$ of  $(\cG, Y)$ to the graph $X$. 

\begin{theorem}[Thm 13, pg. 55, \cite{Serre}]\label{thm:BassSerre2}
With the above notation and hypothesis, the following properties are equivalent: \\
(i) $X$ is a tree.\\
(ii) $\psi: T \to X$ is an isomorphism of graphs. \\
(iii) $\varphi: \pi  \to G$ is an isomorphism of groups. 
\end{theorem} 
 
One of the consequences of Theorem \ref{thm:BassSerre2} is that if a group $G$ acts freely on a tree then $G$ is a free group. 
We can also conclude the following.
 
\begin{corollary}[Cor 1, pg 212, \cite{Cohen}]\label{cor:Cohen}\label{cor:Subgroup}  If $H$ is a subgroup of the fundamental group $\pi (\cG, Y)$ of a graph of groups such that the intersection of $H$ with every conjugate of subgroups $i_v(G_v)$ is the trivial group, then $H$ is free.
 \end{corollary}
 
 In the situation described before Theorem \ref{thm:BassSerre2}, even for an arbitrary graph $X$, the homomorphism $\varphi : \pi \to G$ is surjective if $X$ is connected (see \cite[Lemma 4, pg 34]{Serre}). From the way $\varphi$ is defined it is easy to see that the kernel of $\varphi$ meets every conjugate of a vertex group in the trivial group, hence by Corollary \ref{cor:Cohen}, the kernel of $\varphi$ is a free group.  
 
 
\subsection{Graphs of groups realizing fusion systems}\label{subsect:Realizing}

Given a saturated fusion system $\cF$ defined on a finite $p$-group $S$, there are two different constructions 
of a discrete group $\pi$ with Sylow $p$-subgroup $S$, due to  Leary and Stancu \cite{LearyStancu} 
and Robinson \cite{Robinson}, such that $\cF_S(\pi)=\cF$. In both of these constructions
the group $\pi$ is the fundamental group of a graph of groups. We first state the result by 
Leary and Stancu, which defines $\pi$ as an iterated HNN-extension of the group $S$ and does not 
require the fusion system $\cF$ to be saturated; it works for any fusion system.

\begin{theorem}[Leary and Stancu, \cite{LearyStancu}]\label{thm:LSModel} 
Let $\mathcal{F}$ be a fusion system on a $p$-group $S$ generated by isomorphisms $f_i:P_i\rightarrow Q_i$ 
for $1\leq i\leq r$.  We define a graph of groups $(\mathcal{G},Y)$ 
where $Y$ is the graph having only one vertex $v$ and edges $e_1,\overline{e_1},e_2,\overline{e_2},...,e_r,\overline{e_r}$. 
We define the vertex group $G_v:=S$ and edge groups $G_{e_i}=G_{\overline{e_i}}:=P_i$. The morphisms 
$\phi_{e_i}:P_i\hookrightarrow S$ are the inclusions and the morphisms $\phi_{\overline{e_i}}:P_i\rightarrow S$ 
are $f_i$ composed with inclusions of $Q_i$ into $S$. Then the fundamental group 
$\pi:=\pi(\cG, Y)$ realizes the fusion system $\cF$, that is $\mathcal{F}=\mathcal{F}_S(\pi).$
\end{theorem}

As an example of the Leary-Stancu model, consider the group 
$$G=S_3=\langle a, b \, | \, a^2=b^3=1, aba=b^2 \rangle$$ at prime $p=3$.
The unique Sylow $3$-subgroup of $G$ is $S=\langle b \rangle \cong C_3$. The morphism 
$f: S \to S$ defined by $f(b)=b^2$ generates the fusion system $\cF_S(G)$.
In this case the Leary-Stancu model gives the infinite group 
$$\pi =\langle b, t \, | \, b^3=1, tbt^{-1}=b^2 \rangle \cong C_3 \rtimes \bZ.$$
We will come back to this example later in Example \ref{ex:1} when we discuss cohomology 
of infinite group models. 

We now describe the Robinson model for realizing fusion systems.
 
\begin{theorem}[Robinson \cite{Robinson}]\label{thm:RobinsonModel} 
Let $\mathcal{F}$ be a fusion system on a $p$-group $S$ generated by the images 
$\mathcal{F}_{S_i}(G_i)$ under injective group homomorpisms $f_i:S_i\hookrightarrow S$ 
for $1\leq i\leq r$. We define a graph of groups $(\mathcal{G},Y)$, where $Y$ has vertices 
$v_0,v_1,v_2,...,v_r$ and edges $e_i,\overline{e_i}$ between $v_0$ and $v_i$ for $1\leq i \leq r$. 
The vertex groups are $G_{v_0}:=S$ and $G_{v_i}=G_i$ for $1\leq i \leq r$. The edge groups are 
$G_{e_i}=G_{\overline{e_i}}:=S_i$ and monomorphisms $\phi_{e_i}:S_i\hookrightarrow S$,  
$\phi_{\overline{e_i}}:S_i\hookrightarrow G_i$ are inclusions. Then the fundamental group
$\pi:=\pi (\mathcal{G},Y)$ realizes the fusion system $\cF$.
\end{theorem}

Robinson's theorem is proved also in \cite[Thm 3]{LearyStancu}. Note that to apply Robinson's model 
to a particular fusion system, we need to start with  a collection of subgroups $G_i$ such that 
images of $\cF_{S_i} (G_i)$ generate the fusion system $\cF$. Such a collection always exists
for a saturated fusion system, but does not exist for an arbitrary fusion system (see \cite[Section 4]{LearyStancu}). 

Given a saturated fusion system $\cF$, a family of subgroup $\{P_i\}$ in $S$ is called a \emph{conjugation family} if $\cF$ is generated by 
morphisms in the normalizer fusions system $N_{\cF} (P_i )$. By a theorem of Goldschmidt \cite{Goldschmidt}, the family of $\cF$-centric and $\cF$-radical subgroups in $S$ form a conjugation family. For each $i$, the normalizer fusion system $N_{\cF} (P_i)$ is realized by a finite group $G_i$
with a Sylow $p$-subgroup isomorphic to $N_S(P_i)$ (see \cite[Thm 3.70]{Craven}). Hence if we take the groups $G_i$ in Theorem \ref{thm:RobinsonModel}  as these model groups and the subgroups $S_i$ as their Sylow $p$-subgroups, then the infinite group  $\pi$ obtained using the Robinson construction will realize the fusion system $\cF$.

\begin{remark}\label{rem:RobinsonModel} If $\cF$ is a fusion system realized by a finite group $G$ with a Sylow $p$-subgroup $S$, we take the 
subgroups $G_i$ in the Robinson model as the normalizers $G_i=N_G(P_i)$ where $\{ P_i \}$ is the family of all fully normalized, $\cF$-radical, and $\cF$-centric subgroups of $S$. For the edge groups we take the Sylow $p$-subgroups $S_i=N_S(P_i)$ for every $i$.  
When we refer to the Robinson model for a fusion system realized by a finite group $G$, we will always assume 
that the collection of groups $\{G_i\}$ and $\{ S_i\}$ appearing in Theorem \ref{thm:RobinsonModel} are chosen as 
described here.
\end{remark} 

Associated to a saturated fusion system $\cF$ there is a \emph{centric linking system} $\cL$ (see \cite[Def 9.35]{Craven}) and the triple 
$(S, \cF, \cL)$ is called a 
$p$-local finite group.  The cohomology of the \emph{$p$-local finite group} $(S, \cF, \cL)$ is defined to be the cohomology of the $p$-completion
of the realization $|\cL|$ of the linking system $\cL$. It is shown in \cite[Thm B]{BLO2} that the cohomology of a $p$-local finite group is isomorphic to the subalgebra of $\cF$-stable elements in $H^*(S; \bF_p)$, denoted by $H^* (\cF; \bF_p)$. We define \emph{the cohomology of the fusion system $\cF$} as the inverse limit $$H^* (\cF; \bF_p ) :=\lim _{P \in \cF} H^* (P; \bF_p)$$ and it is easy to see that these two definitions for $H^*(\cF; \bF_p)$ coincide.  

In general the cohomology of a fusion system may be different than the cohomology 
of an infinite group $\pi$ that realizes it. The following theorem by Libman and Seeliger \cite[Thm 1.1]{LibmanSeeliger} 
explains the relation between these two cohomology groups. 

\begin{theorem}[Libman and Seeliger, \cite{LibmanSeeliger}]\label{thm:LibmanSeeliger} 
Let $\cF$ be a saturated fusion system defined on a finite $p$-group $S$, and let $\pi$ be an infinite group realizing $\cF$, 
constructed using the Leary-Stancu model or the Robinson model. Then the map 
$res_S^\pi :H^*(\pi,\mathbb{F}_p)\rightarrow H^*(S,\mathbb{F}_p) $ splits as an $\bF_p$-algebra map and 
has an image isomorphic 
to $H^{*}(\mathcal{F};\mathbb{F}_p) $ that gives 
$$H^{*}(\pi;\mathbb{F}_p)\cong H^{*}(\mathcal{F};\mathbb{F}_p)\oplus \ker(\res_S^\pi).$$
\end{theorem}

The proof of this theorem uses results from the homotopy theory of linking systems. We will give later another proof for this theorem for fusion systems realized by a finite group.
 

\section{Graphs of groups with a storing homomorphism}\label{sect:Storing}

In this section we use the definitions and notation introduced in the previous section. 
 
\begin{definition}\label{def:Storing}
Let $(\mathcal{G}, Y)$ be a graph of groups and $G$ be a finite group. 
A group homomorphism $\chi: \pi(\mathcal{G}, Y) \rightarrow G$ is called a \emph{storing homomorphism}
if it is surjective and for any vertex group $G_v$ with inclusion map $i_v : G_v\rightarrow  \pi(\mathcal{G}, Y)$, the composition 
$\chi \circ i_v : G_v \rightarrow G$ is injective.     
\end{definition}

The kernel of a storing homomorphism $\chi : \pi \to G$ has a trivial intersection with $i_v(G_v)$ for each vertex group $G_v$, hence 
by Corollary \ref{cor:Subgroup}, the kernel of $\chi$ is a free group. Therefore, we have 
an exact sequence  
$$1\rightarrow F\rightarrow \pi (\mathcal{G}, Y) \maprt{\chi}G \rightarrow 1$$ 
where $F$ is a free group. This gives a $G$-action on the abelianization $F_{ab}=F/{[F,F]}$ 
induced by conjugation in $\pi (\cG, Y)$.  

By Theorem \ref{thm:BassSerre1}, the fundamental group $\pi:=\pi (\mathcal{G}, Y)$ acts on a tree 
$T$ without inversion in such a way that the isotropy subgroups of the vertices of $T$ are conjugate to the vertex groups 
$G_v$ of $(\cG, Y)$. The $\pi$-action on $T$ induces an action of $G\cong \pi (\mathcal{G}, Y) / F$ 
on the quotient graph $X=T/F$. From this we obtain a $G$-action on $H_1(X)$.

\begin{lemma}\label{lem:ZGIsom}
There is a $\mathbb{Z}G$-module isomorphism between  $F_{ab}$ and $H_1(X)$.
\end{lemma}

\begin{proof}
Let  $\pi : T\rightarrow X=T/F$ denote the quotient map which takes a point $t\in T$ to its 
$F$-orbit $Ft$. Fix a vertex $v \in T$, and let $\bar{v}=\pi(v)$. By covering space theory, there is an isomorphism 
$F \cong \pi _1 (X, \bar{v})$ given by the map $ \phi : F\rightarrow \pi_1(X, \bar{v} )$ that takes an 
$f\in F$ to the path homotopy class $[\pi (\tau )]$, where $\tau=p(v, f v)$ is a path from $v$ to $fv$. 
	
Let $\hat{\phi}$ be the induced isomorphism between the abelianization groups 
$F_{ab} \to H_1(X)\cong (\pi_1(X,\bar{v}))_{ab}$. We have a commutative diagram    
\[ \begin{tikzcd}
F \arrow{r}{\phi} \arrow[swap]{d}{j} & \pi_1(X,\bar{v}) \arrow{d}{k} \\%
F_{ab} \arrow{r}{\hat{\phi}}& H_1(X)
\end{tikzcd} \]
where $j$ and $k$ are the abelianization maps. To show that $\hat{\phi}$ is a $\mathbb{Z}G$-module isomorphism, 
it is enough to show that for every  $f \in F$ and $g\in G$, the equality $k ( \phi  (\gamma f \gamma ^{-1})) = 
g \cdot k (\phi (f ))$ holds for every $\gamma \in \pi (\cG, Y)$ such that $\chi (\gamma)=g$.
	
We have $\phi(f) = [\pi (\tau)]$ where $\tau =p(v , f v)$ is a path from $v$ to $fv$. For $\gamma \in \pi (\cG, Y)$, we have
 $\phi(\gamma f\gamma^{-1}) = [\pi p (v,\gamma f\gamma^{-1}v) ]$. Note that
$$p (v,\gamma f\gamma^{-1}v) \simeq p (v,\gamma v) \cdot p (\gamma v, \gamma fv) \cdot p (\gamma fv, \gamma f\gamma^{-1}v) $$
in $T$. 
Since $\pi$ annihilates the $F$-action, we have
$$\pi(p(\gamma fv, \gamma f\gamma^{-1}v)  )= \pi(\gamma f \gamma^{-1}p(\gamma v,v))=\pi(p(\gamma v,v)) 
=\pi (p(v, \gamma v ))^{-1}.$$
This gives
\begin{align*}
k (\phi(\gamma f\gamma^{-1}))&=k[\pi p (v,\gamma v))]+k[\pi p (\gamma v,\gamma fv))]-k[\pi p  (\gamma v,v))]=k [\pi p (\gamma v,\gamma fv)]
\end{align*}
in $H_1(X)$. Note that $\pi p (\gamma v, \gamma fv )$ is a loop at $g \bar{v}$ whose homology class is equal 
to $g k[\pi p (v, fv)]$. Hence we have $k(\phi(\gamma f\gamma^{-1}))= gk(\phi(f))$ as desired. We conclude that $\hat{\phi}$ is a 
$\mathbb{Z} G$-module isomorphism between $F_{ab}$ and $H_1(X)$.
\end{proof}

\begin{theorem}\label{thm:Storing}
Let $(\mathcal{G}, Y)$ be a graph of groups and let $\pi:=\pi (\cG, Y)$.
Suppose that $\pi$ has a Sylow $p$-subgroup and that there is a storing homomorphism $\chi: \pi \to  G$ that takes 
a Sylow $p$-subgroup of $\pi$ to a Sylow $p$-subgroup of $G$. Then, there is an isomorphism
$$H^{*-1}(G; \hom(F_{ab}, \bF_p) ) \oplus H^*(G; \bF_p )\cong H^*( \pi ; \bF_p)$$
where $F$ is the kernel of $\chi$.
\end{theorem}

\begin{proof}
To simplify the notation we will denote the images of vertex groups $G_v$ and edge groups $G_e$ under $\chi : \pi \to G$ 
also by $G_v$ and $G_e$. Let $F=\ker \chi$ and $T$ be the tree on which $\pi$ acts with isotropy given by $(\cG, Y)$. 
Consider the $G$-action on the 
graph $X=T/F$. Since $T$ is connected, $X$ is also a connected graph.  

The cellular cochain complex for $X$ with coefficients in $R:=\bF_p$ gives an exact sequence of $RG$-modules
\begin{equation}\label{eqn:Sequence}
0\rightarrow R \rightarrow C^0 (X, R) \maprt{\delta ^0} C^1 (X, R)\rightarrow H^1(X;R) \rightarrow 0.
\end{equation}
The $G$-action on $X$ permutes the cells in $X$, hence we have 
$$C^0 (X,R)=\bigoplus \limits_{v\in OV}R[G/G_v] \ \text{ and } \ C^1 (X,R)=\bigoplus \limits_{e\in OE}R[G/G_e] $$
where $OE$ and $OV$ are orbit representative sets for edges and vertices in $X$, respectively. 
Since $\pi$ has a Sylow $p$-subgroup, there exists a vertex group $G_v$ containing a Sylow $p$-subgroup $S$. 
This means in $G$ the subgroup $G_v$ also includes a Sylow $p$-subgroup of $G$. From this we conclude that 
the map $\mu : R\rightarrow C_0(X, R)$ splits since $|G: G_v |$ is not divisible by $p$. 
We can divide the exact sequence in \ref{eqn:Sequence} into two sequences 
\begin{align*}
0\rightarrow R\rightarrow C_0 (X, R)  \rightarrow  K\rightarrow 0
\end{align*}
\begin{align*}
0\rightarrow K \rightarrow  C_1 (X, R) \rightarrow  H^1(X;R)\rightarrow 0
\end{align*}
where the first sequence splits.	By Shapiro's lemma, the first sequence  gives an isomorphism
\begin{equation}\label{eqn:isom}
\bigoplus \limits_{v\in OV} H^*(G_v,R)    \cong  H^*(G;K)  \oplus H^*(G;R).
\end{equation}
From the second short exact sequence, we also obtain a long exact sequence.
By adding $H^*(G;R)$ to two consecutive terms in this sequence and by using the isomorphism in \ref{eqn:isom}, we get 
$$\cdots \rightarrow H^{*-1}(G,H^1(X;R)) \oplus H^* (G; R)  \rightarrow \bigoplus \limits_{v\in OV} H^*(G_v,R) 	
\maprt{(\delta ^0)^*}   \bigoplus \limits_{e\in OE}H^*(G_e,R) \rightarrow$$
$$ \rightarrow H^*(G,H^1(X;R)) \oplus H^{*+1} (G; R)  \rightarrow	 \cdots $$
The group $\pi =\pi (\cG, Y)$ acts on a tree $T$ with the same isotropy subgroups as the $G$-action on $X$. This gives
a similar long exact sequence from the $\pi$-action on $T$:
$$\cdots  \rightarrow H^*(\pi;R) \rightarrow  \bigoplus \limits _{v\in OV} H^*(G_v ; R )  
\rightarrow \bigoplus \limits_{e\in OE}H^* (G_e,R)  \rightarrow H^* (\pi ; R)\rightarrow \cdots $$
Since the maps in the middle coincide, by the five-lemma we obtain an isomorphism	
$$H^{*-1}(G,H^{1}(X;R)) \oplus H^*(G,R)\cong H^*(\pi ,R).$$
By Lemma \ref{lem:ZGIsom} we have 
$$H^1 (X; R) =\hom (H_1 (X); R) =\hom (F_{ab}, R).$$
Hence the proof is complete.
\end{proof}

\begin{remark}\label{rem:splitting} Theorem \ref{thm:Storing} is a generalization of \cite[Lemma 3.1]{BensonWilson}. 
The proof we give here is very similar to the proof in \cite{BensonWilson}.  There is an alternative approach to proving 
Theorem \ref{thm:Storing} using the Lyndon-Hochschild-Serre spectral sequence  \cite[Thm 6.3]{Brown} 
for the extension $$1 \to F \to \pi \to G \to 1.$$ 
We use this approach later in the proof of Theorem \ref{thm:thmC}.

\end{remark}

In Theorem \ref{thm:Storing}, the assumption that the fundamental group $\pi$ has a Sylow $p$-subgroup is necessary, as 
the following example illustrates.

\begin{example} Let $\pi =C_2 \ast C_2=\langle a_1, a_2 \, | \, a_i ^2, \text{ for } i=1,2  \rangle$. 
Then $G$ has no Sylow $2$-subgroup since the subgroups $\langle a_1 \rangle$ and $\langle a_2 \rangle$ are
not conjugate to each other in $\pi$. Note that if we take $G=C_2\times C_2$ and define the storing homomorphism
$\chi : \pi \to G$ by taking $a_1$ and $a_2$ to the generators of $G$, then the kernel of $\chi$ is the subgroup 
$F=\langle (a_1 a_2)^2 \rangle \cong \bZ$. We have $H^* (G; \bF_2 )\cong \bF_2 [x_1, x_2]$
where $\deg x_i=1$, and $H^i (\pi ; \bF_2 )\cong \bF_2\oplus \bF_2 $ for all $i\geq 1$. Hence, the isomorphism in 
Theorem \ref{thm:Storing} does not hold in this case. 
\end{example}

Now we are ready to prove Theorem \ref{thm:thmA}.

\begin{proof} Let $G$ be a finite group with Sylow $p$-subgroup $S$. Suppose that $G$ has no proper subgroups
that control $p$-fusion in $G$. Let $\pi _{LS}=\pi (\cG, Y)$ denote the infinite group realizing the fusion system 
$\cF :=\cF_S(G)$ constructed according to the Leary-Stancu model, as explained in Theorem \ref{thm:LSModel}. 
Let $\chi : \pi _{LS} \to G$ denote the group homomorphism that takes $S\leq \pi_{LS}$ to $S \leq G$ and the generators
$t_i$ to the  group elements $g_i \in G$ where $g_i$ is an element in $G$ such that $c_{g_i} =f_i : P_i \to Q_i$ for $i=1,\dots ,r$. 
Note that the image of $\chi$ controls $p$-fusion in $G$, hence by our assumption above $\chi$ is surjective.
The only vertex group of $\cG$ is $S$, and the restriction of $\chi$ to $S$ is injective, hence $\chi$ is a storing 
homomorphism. 

Now let $\pi_R$ denote the infinite group $\pi _R$ obtained using the Robinson model with vertex 
groups $N_G(P_i)$ for a collection of $p$-centric subgroups $\{P_i\}$, where $Y$ is a star-shaped graph with the center 
having vertex group $S$. In this case the group $\pi$ is generated by the subgroups $i_v (N_G(P_i))$ in $\pi$, 
so $\chi$ is defined as a map that takes the subgroups $i_v(N_G(P_i))$ injectively to the subgroups $N_G(P_i)$ 
in $G$. It is easy to see that in this case too, $\chi$ is a storing homomorphism. 
  
In both cases there is a storing homomorphism $\chi : \pi \to G$. Since the kernel of a storing homomorphism is 
a free group, this gives a group extension $1 \to F \to \pi \to G \to 1 $, where $F$ is a free group. Applying 
Theorem \ref{thm:Storing}  to the storing homomorphism $\chi : \pi \to G$ gives the isomorphism  in the statement of the theorem.
\end{proof}

As a corollary of Theorem \ref{thm:thmA}, we obtain the following.

\begin{corollary} Let $G$ and $\pi$ be as in Theorem \ref{thm:thmA}. Then the kernel of the restriction map
$\res^\pi _S : H^*(\pi ; \bF_p) \to H^* (S; \bF_p)$ is isomorphic  to $H^{*-1} (G ; \hom (F_{ab}; \bF_p ))$.  
\end{corollary}

\begin{proof} The image of the restriction map 
$\res^\pi _S$ is isomorphic $H^*(\cF ; \bF_p) \cong H^*(G; \bF_p)$.   
Hence Theorem \ref{thm:thmA} gives the desired isomorphism.  
\end{proof}


\section{An Infinite Family of Examples}\label{sect:Examples}

We start with an easy example to illustrate that infinite groups obtained using the Leary-Stancu model 
may have cohomology groups that are not isomorphic to the cohomology of the fusion systems that they realize.

\begin{example}\label{ex:1}
Let $G=S_3=\langle a, b \, | \, b^3=a^2=1, aba=b^2 \rangle$ and $R=\bF _3$. The Sylow $3$-subgroup of $G$ is 
$S=\langle b \rangle \cong C_3$. The Leary-Stancu model is the infinite group 
$$\pi =\langle b, t \, | \, b^3=1, tbt^{-1}=b^2 \rangle \cong C_3 \rtimes \bZ.$$
The storing homomorphism $\chi : \pi \to G$ takes $t \in \pi$ to $a\in G$, so the kernel of $\chi$ is 
$F=\langle t^2 \rangle \cong \bZ$. The $G$-action on $F$ is trivial, hence Theorem \ref{thm:Storing} gives that 
\begin{equation}\label{eqn:Ex1Isom}
H^*(\pi ; \bF_3) \cong H^* (S_3; \bF_3) \oplus H^{*-1} (S_3;\bF_3).
\end{equation}
The cohomology ring of $C_3$ is $H^*(C_3; \bF_3 )= \bigwedge_{\bF_3} (x) \otimes \bF_3 [y]$, where $\deg x=1$ 
and $\deg y=2$, and the cohomology ring of $S_3$ is the subalgebra $$H^* (S_3; \bF_3) = \bigwedge \nolimits _{\bF_3} (xy)\otimes \bF_3 [y^2].$$
We can calculate the cohomology of $\pi$ using the sequence $1 \to C_3 \to \pi \to \bZ \to 1$. The LHS-spectral sequence 
has only two nonzero vertical lines and $d_1: H^i (C_3; \bF_3) \to H^i (C_3, \bF_3)$ is identity only at dimensions $i$ 
where $i \equiv 1,2 $ mod $4$.  From this calculation, we can easily see that $H^* (\pi; \bF_3) \not \cong H^* (S_3; \bF_3)$
and $\ker( \res^{\pi } _S ) \cong H^{*-1} (S_3;  \bF_3)$. \qed
\end{example}

We now give another example of storing homomorphisms $\pi \to S_3$ where $\pi$ is an amalgamation 
of two finite groups and $S_3$ acts nontrivially on $F$.

\begin{example}\label{ex:2}
Let $\pi =S_3 \ast _{C_3} S_3$, and $R=\bF _3$. We can give a presentation for $\pi$ as follows:
$$\pi =\langle b, a_1, a_2 \, | \, a_i ^2=1, a_i b a_i=b^2 \text{ for } i=1,2 \rangle.$$
The group $G=S_3$ is a store of $\pi$ with storing homomorphism which takes both $a_1$ and $a_2$ 
to $a \in G$. The kernel of $\chi$ is $F=\langle a_1a_2 \rangle \cong \bZ$. In this case $G=S_3$ acts 
nontrivially on $F$ since $a_1 (a_1a_2)a_1=a_2a_1=(a_1a_2)^{-1}$. We usually denote this one-dimensional 
$\bZ G$-module by $\widetilde \bZ$. By Theorem \ref{thm:Storing}, we have $$H^*(\pi ; \bF_3 ) \cong 
H^* (S_3;  \bF_3 ) \oplus H^{*-1} (S_3 ; \widetilde \bF_3 ).$$ 
We can calculate the cohomology groups $H^* (S_3; \widetilde \bF_3) $ using the sequence of $\bF_3 G$-modules 
$0 \to \bF_3 \to \bF_3[G/S]\to \widetilde \bF_3 \to 0$. We obtain that $H^* (S_3; \widetilde \bF_3 ) \cong \bF_3 $ 
for $i\equiv 1,2 $ mod $4$, and $0$ otherwise.   The cohomology of $\pi$ can be calculated using the 
long exact sequence for groups acting on a tree. From these we can verify that the isomorphism above holds.
In this case the kernel of the restriction map to the Sylow $3$-subgroup is $H^{*-1} (S_3; \widetilde \bF_3)$. \qed
\end{example}

In the rest of this section we consider the $2$-fusion system of the group $GL(n, 2)$ and show that
the cohomology of the infinite group $\pi _R$ constructed using the Robinson model is not isomorphic to the cohomology 
of the $2$-fusion system for $n \geq 5$.  This gives an infinite family of examples with this property.
Examples of groups with this property were already known. In \cite[Prop 6.8]{Seeliger}, 
it is shown that the mod-$2$ cohomology of $G=C_2^3 \rtimes GL(3, 2)$ is not isomorphic to the cohomology
of $\pi_R$.

To construct an infinite group using the Robinson model for the $2$-fusion system of $GL(n, 2)$, we must understand 
all fully normalized, $\cF$-radical, $\cF$-centric subgroups for the fusion system $\cF=\cF_S(G)$, where $G=GL(n, 2)$ and
$S$ is a Sylow $2$-subgroup of $G$. Since $GL(n, p)$ is an algebraic group we will quote some standard
results from \cite[Sec 6.8]{Benson2} and \cite{MalleTesterman} to describe its $p$-radical and $p$-centric subgroups.
We also refer to \cite[Appendix B]{LibmanViruel} for some of the results below.

Let $S$ be the subgroup of $G=GL(n, 2)$ consisting of the upper triangular matrices. Since the order 
of $S$ is $2^{(n-1)(n-2)/2}$ and the order of $G$ is $(2^n-1)(2^n-2)\cdots (2^n-2^{n-1} )$, the index $|G:S|$ is odd. 
Hence $S$ is a Sylow $2$-subgroup of $G$. The Borel subgroups of $G$ are the conjugates of $S$ and $N_G(S)=S$
(see \cite[Thm 6.12]{MalleTesterman}). Parabolic subgroups of $G$ are stabilizers of flags $0=V_0 < \cdots < V_k=\bF_p ^n$,
so every parabolic subgroup is conjugate to a subgroup $N$ consisting of matrices of the form 
$$\left[\begin{array}{c|c|c}
* & * & * \\ 
\hline
0 & * & * \\
\hline
0 & 0 & * \\
\end{array}\right] 
$$
The unipotent radical of $N$ is the subgroup $U$ of matrices of the form 
$$\left[\begin{array}{c|c|c}
I & * & * \\
\hline
0 & I & * \\
\hline
0 & 0 & I \\
\end{array}\right] 
$$

We now state a special case of the Borel-Tits theorem (see \cite[Thm 6.8.4]{Benson2}) to identify the $p$-radical subgroups
of $G$.

\begin{theorem}[Borel-Tits]\label{BorelTit}
If $G=GL(n, p)$ then a $p$-subgroup $U$ is $p$-radical if and only if $N_G(U)$ is parabolic 
and $U$ is its unipotent radical.
\end{theorem}

We also have the following observation.

\begin{lemma}\label{lem:Centric}
Let $S$ be the group of upper triangular matrices in $G=GL(n, 2)$ and $\cF=\cF_S(G)$. 
Then any unipotent radical $U$ of a parabolic group $P$ containing $S$ is $\cF$-centric.
\end{lemma}

\begin{proof} If $U$ is $\cF$-centric and $V \geq U$, then $V$ is also $\cF$-centric. 
We know that the maximal parabolic subgroup corresponds to the minimal unipotent radicals. Then, it is enough 
to prove that the statement holds for all maximal parabolic subgroups containing $S$.
	
Take any maximal parabolic subgroup $N_G(U)$ containing $S$, where $U$ is a subgroup of 
the form 	
\[
U_m=\left[ {\begin{array}{cc}
I_m & M_{m,n-m}(\mathbb{F}_2) \\
0 & I_{n-m} \\
\end{array} } \right]      
\]
for some $m$. Then $U \lhd S$, hence $U$ is fully $\cF$-normalized. This implies that fully 
$\cF$-centralized, and hence it is enough to show that $C_S(U)=Z(U)$.
Take any $s\in S$ centralizing $U_m$. If we write	
\[
s=\left[ {\begin{array}{cc} A & B \\ 0 & C \\
\end{array} } \right]       
\]
where $A$ and $C$ are upper triangular matrices with diagonal entries equal to $1$, then the equation $su=us$ 
gives that we must have $AM=MC$ for any  $M\in  M_{m,n-m}(\bF_2) $.
Fix any $1\leq i \leq m$ and $1\leq j \leq m-n$. Choosing $M$ to have all entries 0 except the $(i,j)$-th entry, which is equal to 1, 
the equality $AM=MC$ gives that $c_{j,k}=0$ for $k\neq j$ and $a_{l,i}=0$ for $l\neq i$.  This gives $A=I_m$ 
and $C=I_{n-m}$. Hence $s$ lies in $U_m$. We conclude that $U$ is $\cF$-centric.
\end{proof}

The argument above can be extended to show that $C_G(U)=Z(U)$ for every unipotent radical $U$ normal in $S$.
This gives that $C_G'(U)=1$ and $N_G(U)/U=N_G(U)/C_G(U)U$ for these subgroups. In particular, $U$ is $\cF$-radical
since $U$ is $p$-radical in $G$. We conclude the following.

\begin{theorem}\label{U is F-alperin iff NU is parabolic}
Let $G=GL(n, 2)$. The subgroup of upper triangular matrices $S$ in $G$ is a Sylow 2-subgroup of $G$. 
Let $\mathcal{F}=\mathcal{F}_S(G)$. Then $U$ is fully normalized, $\cF$-radical, $\cF$-centric 
subgroup of $S$ if and only if $N_G(U)$ is parabolic containing $S$ and $U$ is its unipotent radical.
\end{theorem}

\begin{proof}
The first sentence is explained above.  Let $U$ be a fully normalized, $\cF$-centric, and $\cF$-radical subgroup in $S$. 
By Lemma \ref{lem:CentricRadical}, an $\cF$-centric, $\cF$-radical subgroup of $S$ is $p$-radical, $p$-centric in $G$.  
Hence by Theorem \ref{BorelTit}, $N_G(U)$ is parabolic 
and $U$ is its unipotent radical. Since $N_G(U)$ is parabolic, $N_G(U) \supset B$ for some Borel subgroup $B$. 
Since Borel subgroups are conjugate, there exists $g\in G$ such that $S=gBg^{-1}$. Let $P=gUg^{-1}$. Then 
$N_G(P)=gN_G(U)g^{-1}\supset gBg^{-1}=S$. Since $U$ is fully normalized, we have 
$|N_S(U)|\geq |N_S(P)|$. So $N_S(P)=S$ gives that $N_S(U)=S$, which means $N_G(U)$ contains $S$ as desired.
	
For the other direction, assume that $U$ is a subgroup of $S$ such that $N_G(U)$ is parabolic containing $S$ and $U$ is 
its unipotent radical. By Theorem \ref{BorelTit}, $U$ is $p$-radical. By Lemma \ref{lem:Centric}, $U$ is $\cF$-centric.
By the remark after the proof of Lemma \ref{lem:Centric},
$U$ is $\mathcal{F}$-radical. Since $N_S(U)=S$, we can also say that $U$ is fully normalized.  
\end{proof}

%
%

Now we are ready to prove Theorem \ref{thm:thmB}

\begin{proof}[Proof of Theorem \ref{thm:thmB}]
By  \cite[Table 6.1.3]{GLS}, we have  $H^2(\cF; \bF_2 )=H^2(G; \bF_2)=0$ for $n\geq 5$. We will 
show that $H^2(\pi; \bF_2)\neq 0.$ The vertex groups of $\pi$ are the subgroup $S$ and the normalizers $N_G(P_i)$ of 
fully normalized, $\cF$-radical, $\cF$-centric subgroups $P_i$ of $S$. From 
Theorem \ref{U is F-alperin iff NU is parabolic},  the vertex groups of $\pi$ are $N_0=S$ and the parabolic 
subgroups $N_1, N_2,..., N_k$ containing $S$. This gives a long exact sequence 
\begin{equation}\label{eqn:GLn2}
\cdots \rightarrow \bigoplus \limits_{0\leq i \leq k} H^1(N_i;\mathbb{F}_2)  \xrightarrow {f} 
\bigoplus \limits_{1}^{k} H^1(S;\mathbb{F}_2) \rightarrow H^2(\pi ;\mathbb{F}_2)\rightarrow \cdots
\end{equation}
Note that for any $i$, we have  $|H^1(N_i;\mathbb{F}_2)|\leq |H^1(S;\mathbb{F}_2)|$ because $S$ is a 
Sylow 2-subgroup of $N_i$. Without loss of generality, assume that $N_1, N_2, \dots, N_{n-1}$ are maximal 
parabolic subgroups such that, for $1\leq m \leq n-1$, we have
\[
P_m=\left[ {\begin{array}{cc}
GL(m,2) & M_{m,n-m}(\mathbb{F}_2) \\
0 & GL(m-n,2) \\
\end{array} } \right].
\]
Then we have that $N_1\cong N_{n-1}\cong C_2^{n-1} \rtimes GL(n-1,2) $. Note that $H^1(N_1 ;\mathbb F_2)
\cong \hom(N_1, C_2).$ Take any $\phi \in \hom(N_1, C_2)$. The restriction of $\phi$ to $GL(n-1,2)$ 
is the zero homomorphism because $GL(n-1, 2)$ is a simple group. If $\phi$ is non-zero, then we have 
$\phi(a)=1$ for some $a\in C_2^{n-1}$. Take any nonzero $ b\in C_2^{n-1}$ such that $b \neq a$. 
Since $GL(n-1,2)$ acts on $C_2^{n-1}$ by conjugation and it sends any nonzero element to a
nonzero element, we have $\phi(a)=\phi(b)=\phi(a+b)=1$ which is a contradiction. We conclude 
that
$$H^1(N_1;\mathbb{F}_2)=H^1(N_{n-1};\mathbb{F}_2)=0.$$
From this we obtain that
$$\sum_{0\leq i \leq k} \dim H^1(N_i;\mathbb{F}_2) <  k \dim H^1(S;\mathbb{F}_2) $$
because in the left hand side two terms are zero as shown above, and for all the other summands we have 
$| H^1(N_i;\mathbb{F}_2) | \leq | H^1(S;\mathbb{F}_2)|$. 
This gives that the map $f$ in the long exact sequence \ref{eqn:GLn2}  is not surjective.  
Hence $H^2(\pi ;\mathbb{F}_2)\neq 0.$ This completes the proof.
\end{proof}


\section{Realizing fusion systems via group actions on graphs}\label{sect:GroupActions}

Let $G$ be a finite group acting on a connected graph $X$ without inversion.  As we discussed 
in Section \ref{subsect:GroupActions}, using the isotropy subgroups of the vertices and edges 
of $X$, we can define a graph of groups $\cG$ on the graph $Y=X/G$. Let $\pi :=\pi (\cG, Y)$ denote 
the fundamental group of this graph of groups. The map $\chi : \pi \to G$ defined by sending the vertex 
groups $G_v$ of $\pi$ to the corresponding isotropy subgroups in $G$ is a storing homomorphism. 
The surjectivity of $\chi$ follows from the fact that $X$ is connected. We also know that the kernel 
of $\chi$ is a free group by Corollary \ref{cor:Subgroup}. This gives an extension of groups $$1 \to F \to \pi \to G \to 1.$$

There is an alternative description of the group $\pi =\pi (\cG, Y)$ that is associated to a $G$-action on a graph $X$. 
Consider the Borel construction $EG\times _G X$. From the description of the total space of $(\cG, Y)$ it is easy 
to see that the total space $X(\cG, X)$ is homotopy equivalent to the Borel construction $EG\times _G X$ 
(see \cite[pg. 167]{ScottWall}). Hence $\pi =\pi_1 (EG\times _G X)$. The Borel construction gives a fibration 
$$X \rightarrow EG\times_GX\rightarrow BG$$
that induces a long exact sequence in homotopy groups
$$\cdots \rightarrow\pi_2 (BG)\rightarrow \pi_1(X) \rightarrow \pi_1(EG\times_GX)\rightarrow \pi_1(BG)\rightarrow \pi_0(X)\rightarrow\cdots$$
Since $X$ is connected and $BG$ is a classifying space of a finite group, we obtain a short exact sequence
$$1 \rightarrow \pi_1X \rightarrow \pi_1(EG\times_GX)\rightarrow \pi_1(BG)\rightarrow 1.$$
This shows that the map $\chi:\pi\rightarrow G$ is surjective and its kernel is isomorphic
to $\pi _1(X)$, which is a free group.

Note that infinite groups obtained in this way may not have a Sylow $p$-subgroup in general.

\begin{example} Consider the $G=C_2=\langle a \rangle$ action on a circle $X$ with 
the action $g (x,y)= (x, -y)$. We can view $X$ as the realization of the graph with two vertices and two edges. Then the quotient graph $Y$ is a graph with single edge and two vertices, where vertex groups are $G_v=C_2$ at both vertices, and the edge group is $1$. The fundamental group $\pi:=\pi (EG \times _G X)$ is the free product $C_2\ast C_2$ which is isomorphic to the infinite Dihedral group $D_{\infty}$. In this case $\pi$ does not have a Sylow $2$-subgroup.    \qed
\end{example}

We can give a list of conditions on the $G$-action on $X$ to guarantee the existence of a Sylow $p$-subgroup  in $\pi$.
 
\begin{proposition}\label{prop:SylowExists}
Let $G$ be a finite group with a Sylow $p$-subgroup $S$, and let $X$ be a connected $G$-graph 
where $G$ acts without inversion. If there is a vertex $v_0$ of $X$ such that
\begin{enumerate}
\item $S$ fixes $v_0$, and
\item for any vertex $v\in X$ there exists a path $y_1, \dots, y_n$ from $v_0$ to $gv$ for some $g\in G$ such that 
for each $i$, a Sylow $p$-subgroup of the stabilizer group $\Stab_G ( t(y_i))$ lies in the stabilizer of the edge $y_i$,
\end{enumerate}
then the fundamental group $\pi=\pi (\cG, Y)$ associated to the $G$-action on $X$ has a Sylow $p$-subgroup 
that maps isomorphically to $S$ under the storing homomorphism $\chi: \pi \to G$. 
\end{proposition}

\begin{proof}
The vertex groups of $\cG$ are the stabilizers of the $G$-action on $X$. Conditions (1) and (2) in the theorem 
implies the condition (ii) of \cite[Prop 3.3]{LibmanSeeliger}. Condition (i) of this proposition already holds since 
all vertex groups are finite, hence
by \cite[Prop 3.3]{LibmanSeeliger}, $\pi$ has a Sylow $p$-subgroup isomorphic to $S$. 

Since $S$ fixes the point $v_0$, $S$ maps into $\pi$ via the inclusion $i_v: G_{v}  \to \pi$ where $v$ is the image 
of $v_0$ under the quotient map $X\to Y$. It is clear that the storing homomorphism $\chi : \pi \to G$ takes a 
Sylow $p$-subgroup of $\pi$ isomorphically onto $S \leq G$.
\end{proof}

The storing homomorphism $\chi: \pi \to G$ can also be used to compare the corresponding fusion systems.

\begin{theorem}\label{thm:mainthm3}
Let $G$ be a finite group with a Sylow $p$-subgroup $S$, and let $X$ be a connected graph 
on which $G$ acts without inversion. Assume that $X$ has a vertex $v_0$ satisfying conditions
(1) and (2) of Proposition \ref{prop:SylowExists}. Then after identifying the Sylow $p$-subgroups, we have
$\cF_S (\pi) \subseteq  \cF_S (G)$.
Furthermore, if for every $p$-centric, $p$-radical subgroup $P$ in $G$, the normalizer $N_G(P)$ fixes a vertex on $X$, then  
 $\cF_S(\pi) =\cF_S(G)$.
\end{theorem}

\begin{proof}
Let $c_{\gamma} : Q \to R$ be a morphism in $\cF_S (\pi)$, where $\gamma \in \pi$.  Let $g=\chi (\gamma)$. 
After identifications, $c_{\gamma}$ is equal to the
conjugation map $c_g: Q\to R$ which is a morphism in $\cF_S(G)$. Hence $\cF_S(\pi)\subseteq \cF_S(G)$. 

To prove the second statement, let $c_g : P \to P$ be a morphism in $\cF_S(G)$ where $P\leq S$ is a $p$-centric, $p$-radical subgroup in $G$,
and  $g\in N_G(P)$. By the given condition, $N_G(P) \leq \Stab _G(v)$ for some $v$ in $X$. Since $\chi$ is a storing homomorphism, it induces 
an isomorphism between $i_v(G_v) \leq \pi$ and $\Stab_G(v) \leq G$. This means that there is a $\gamma \in i_v(G_v) \leq \pi$ 
such that $c_{\gamma} : P \to P$ is equal to the morphism $c_g: P\to P$, hence $c_g\in \cF_S(\pi)$. Since, by Alperin's theorem
$\cF:=\cF_S(G)$ is generated by morphisms $c_g: P \to P$ where $P$ is fully normalized, $\cF$-centric, and $\cF$-radical.
By Lemma \ref{lem:CentricRadical}, $\cF$-centric, $\cF$-radical subgroups in $\cF_S(G)$ are $p$-centric and $p$-radical in $G$, hence 
we can conclude that $\cF_S(\pi)=\cF_S(G)$.
\end{proof}
 
\begin{remark}\label{rem:Notconnected}  If $X$ is a $G$-graph which is not connected 
but such that the quotient graph  $Y=X/G$ is connected, then there is a subgroup $H \leq G$ formed by elements $g\in G$ such that 
$gX_1=X_1$ for some connected component $X_1$ of $X$. It is easy to see that there is  
an isomorphism $X\cong G\times_HX_1$ of $G$-spaces. This gives $$EG\times_GX \simeq EG\times _G (G \times _H X_1)\simeq 
EH \times_H X_1.$$ If the connected $H$-graph $X_1$ satisfies the conditions in Theorem \ref{thm:mainthm3}, 
then $$\pi := \pi_1(EG\times_GX)=\pi_1(EH\times_HX_1)$$ realizes the fusion $\mathcal{F}_S(H)$.  
In general the fusion system $\cF_S(H)$ will not be equal to $\cF_S(G)$, but in some of the cases we consider below this will be true, 
and we will have applications to Theorem \ref{thm:mainthm3} even when the $G$-graph $X$ is not connected. \qed
\end{remark}

The main example that motivated this section is the action of $G$ on the graph of $p$-centric, $p$-radical subgroups of $G$.

\begin{example}
Let $G$ be a finite group and $X$ be the graph whose vertices are the $p$-centric, $p$-radical subgroups of $G$. There is an edge between $P$ and $Q$ in $X$ if $Q < P$. The group $G$ acts on $X$ by conjugation, and this action is without inversion of edges.  The graph $X$ is not connected in general but the subgroup $H$ that stabilizes a connected component is generated by the stabilizers of vertices of that component. Let $X_1$ be the connected component that includes $S$, then $H$ is generated by normalizers of the $p$-centric, $p$-radical subgroups in $S$. The stabilizers are normalizer subgroups $N_G(P)$, and by Alperin's fusion theorem they generate the fusion system $\cF_S(G)$, so in this case we have $\cF_S(H)=\cF_S(G)$. 

Let $\pi:=\pi_1 (EG\times _G X) \simeq \pi_1(EH \times _H X_1)$. Consider the $H$-action on the connected graph $X_1$. If we take $v_0$ as  $S$, then the stabilizer of $v_0$ in $H$ contains a Sylow $p$-subgroup of $G$. Condition (ii) of Proposition \ref{prop:SylowExists} can 
be checked easily. Let $P$ be any vertex in $X_1$. There is a vertex $R$ that is conjugate to $P$ such that
$R$ is fully normalized and $R \leq S$. Then the path between $R$ and $S$ formed by a single edge satisfies the required condition because
the stabilizer of the edge $R<S$ includes the Sylow $p$-subgroup of $N_G(R)$, which is $N_S(R)$.  So $\pi$ has a Sylow $p$-subgroup that maps isomorphically to the Sylow $p$-subgroup of $G$. The condition of Theorem \ref{thm:mainthm3} also holds because for every $p$-centric, $p$-radical subgroup of $G$, the subgroup $N_G(P)$ is the stabilizer of the vertex $P$. Hence  $\pi$ realizes $\cF_S(G)$.  

Note that the infinite group $\pi$ obtained from this group action is the same as the model given by Libman and Seeliger in \cite[Sec. 4.1]{LibmanSeeliger}, which is different than the Robinson model given in Theorem \ref{thm:RobinsonModel}. This version of Robinson model is interesting from the point of view of the normalizer decomposition of classifying spaces. Let $\cC$ denote the collection of all $p$-centric, $p$-radical subgroups in $G$. The graph $X$ is the 1-skeleton of the poset of subgroups in $\cC$. Since the collection of $p$-centric, $p$-radical subgroups of $G$ is a $p$-ample collection, the projection map $EG\times _G |\cC| \to BG$ induces a mod-$p$ isomorphism (see Definition 7.7 and 8.10 in \cite{Dwyer}). In addition, we have $B\pi\cong EG\times _G X$, hence the map $f: B\pi \to BG$ induced by $\chi: \pi \to G$ gives a mod-$p$ cohomology isomorphism if and only if the inclusion map $i: X\to |C|$ induces a mod-$p$ cohomology isomorphism  $$i^* : H^*(EG\times_G |C| ; \bF_p) \to H^* (EG\times_G X ; \bF_p).$$ It is easy to see that for many groups these two cohomology rings will not be isomorphic, in particular, when the permutation modules for the higher dimensional cells in $|\cC|$ are not free.  
\qed
\end{example}

In the rest of the section we consider the group action on the graph of elementary abelian $p$-subgroups of $G$. 
Through out this discussion, we assume $G$ is a finite group with $p$-rank equal to $2$, meaning that $G$ has 
a subgroup isomorphic to $(\bZ /p)^2$ but it has no subgroups isomorphic to $(\bZ/p)^3$. Let $X =\cA_p (G)$ 
be the poset of all nontrivial elementary abelian $p$-subgroups in $G$. Since $G$ has $p$-rank equal to $2$ 
this is a one-dimensional poset, hence we may consider it as a graph whose vertices are the nontrivial elementary 
abelian $p$-subgroups of $G$ and where there is an edge between $E_1$ and $E_2$ if and only if $E_1 < E_2$. The group $G$ 
acts on $X$ by conjugation. 

Let $\pi :=\pi (\cG, Y)$ denote the fundamental group of the graph of 
groups associated to the $G$-action on $X$.  Note that the vertex groups of $\cG$ are the normalizers $N_G(E)$. 
By the discussion above, the group $\pi$ can also be described as the fundamental group $\pi:=\pi _1 (EG\times _G |\cA_p (G) |)$.  
In this case the mod $p$ cohomology  of $\pi$ is known to be isomorphic to the cohomology of $G$. 
This is a theorem due to P. Webb.

\begin{theorem}[Webb, \cite{Webb}, Thm E] Assume that $G$ is a finite group with $\text{rk}_p(G)=2$. Let $\cA _p (G)$ be the poset of nontrivial elementary abelian $p$-subgroups in $G$, and let $\pi:=\pi_1(EG \times_G |\cA_P (G)|)$. Then $\widetilde H_i (|\cA_p (G)|  ; \bF_p )$ is a projective $\bF_p G$-module for $i=0,1$, and  there is an isomorphism $$H^*(\pi ,\mathbb{F}_p) \cong H^*(G,\mathbb{F}_p).$$
\end{theorem}

\begin{proof} The first part is proved in \cite[Thm E]{Webb}. The isomorphism of cohomology groups is given in \cite[pg. 153]{Webb}.
\end{proof}

Using Theorem \ref{thm:mainthm3} we can also conclude the following.

\begin{theorem}\label{thm:Rank2}
Let  $G$ be a finite group with  $\text{rk}_p(G)=2$, and let $S$ be a Sylow $p$-subgroup of $G$. Then $\pi :=\pi_1(EG \times_G | \cA_p (G) |)$ has a Sylow $p$-subgroup isomorphic to  $S$, and $\mathcal{F}_S(\pi)=\mathcal{F}_S(G) $. 
\end{theorem}

\begin{proof}	
 Since the center $Z(S)$ of $S$ is not trivial, we can take a subgroup $C_1$ of order $p$ which lies in $Z(S)$.
Let $X_1, X_2,\dots ,X_k$ be the connected components of $X:=|\cA_p(G)|$, assume that $C_1\in X_1$. 
Define $$H:=\{g\in G \, |\, gX_1=X_1\}.$$

Since $S$ fixes $C_1$, it also fixes the component $X_1$. Hence, we have $S \leq H$. Any elementary abelian 
$E \leq S$ is connected to $C_1$ by a path in $X_1$ because the elementary abelian subgroup $EC_1$ of $S$ is 
connected to both $E$ and $C_1$. For any subgroup $Q \leq S$, the normalizer $N_G(Q)$ normalizes the largest 
central elementary abelian subgroup $Z_Q$ of $Q$ because $Z_Q $ is a characteristic subgroup of $Q$. This means 
that $N_G(Q)$ fixes the vertex $Z_Q$ in $X$, which lies in $X_1$. Hence, $N_G(Q) \leq H$ for all $Q \leq S$. 
By Alperin's fusion theorem, we obtain that 
$$\mathcal{F}_S(H)=\mathcal{F}_S(G).$$

To complete the proof, it is enough to prove that $\pi \cong \pi _1(EG \times_H X_1)$ realizes the fusion system $\cF:=\cF_S(H)$. 
We argue that in this situation conditions (1) and (2) of Theorem \ref{thm:mainthm3} are satisfied.  For condition (1), we 
choose the vertex $v_0$ as the subgroup $C_1$. It is clear that $S$ fixes $C_1$ since $C_1 \leq Z(S)$. 

For condition (2), take any $E$ in the poset $X_1$. Since the group $E$ fixes  the vertex $E$, it also fixes the component $X_1$, hence $E \leq H$. 
By replacing $E$ with an $H$-conjugate subgroup we can assume $E$ is a fully $\cF$-normalized of subgroup of $S$. Since $\text{rk}_p (G)=2$, the subgroup $E$  has order $p$ or $p^2$. We will analyze these two cases.

We denote the elementary abelian $p$-subgroups of $S$ order $p$ and $p^2$ by $C_i$'s and $E_i$'s, respectively. Since $C_1\leq  Z(S)$, we have  $C_1\leq E_i$ for all $i$ because otherwise the group $C_1E_i$ is an elementary abelian group of order $p^3$, giving a contradiction with $\text{rk}_p(G)=2$.

\begin{figure}[h]
	\begin{tikzpicture}[node distance=2cm]
	\title{poset}
	\node(A4)                           {$C_1$};
	\node(V4)       [below right of=A4] {$E_n$};
	\node(C31)      [below left of=A4]  {};
	\node(C32)      [left of=C31]       {$E_1$};
	
	\node(C22)      [below of=V4]       {$C_y$};
	\node(C21)      [left of=C22]       {$C_x$};
	\node(C23)      [right of=C22]      {$C_z$};
	
	\node(a1)      [below of=C32]       {$...$};
	\node(a2)      [left of=a1]       {$C_2$};
	\node(a3)      [right of=a1]      {$C_r$};
	
	\draw(C32)       -- (a3);	
	\draw(C32)       -- (a2);
	\draw(A4)       -- (V4);
	
	\draw(A4)       -- (C32);
	
	\draw(V4)       -- (C21);
	\draw(V4)       -- (C22);
	\draw(V4)       -- (C23);
	\path (C31) -- (V4) node [midway] {$\cdots$};
	\end{tikzpicture}
	\caption{Poset $\cA_p (G)$ }
\end{figure}

If $E=E_i$ for some $i$, then there is an edge between $C_1$ and $E$. The stabilizer of $E$ has Sylow $p$-subgroup $N_S(E)$ because $E$ is fully $\cF$-normalized. The subgroup  $N_S(E)$ is contained in $N_H (C_1)\cap N_H(E)$, which is the stabilizer of the edge between $C_1$ and $E$. Hence, condition (2) is satisfied in this case.

Now assume that $E=C_i$ for some $i$. Then $C_i$ is contained in $E_j=C_1C_i$. 
Now we consider the path $C_1, E_j, C_i$ for the condition (2). Since $C_i$ is fully $\mathcal{F}$-normalized, $N_H(C_i)$ has Sylow $p$-subgroup $N_S(C_i)$ contained in $N_H(E_j)$ because if $s\in S$ fixes $C_i$ then it fixes $E_j=C_1C_i$. This shows that the edge from $E_j$ to $C_i$ satisfies the Sylow $p$-subgroup condition. 

For the first edge from $C_1$ to $E_j$, if  $E_j$ is also fully $\cF$-normalized then we are done in a similar way. Assume to the contrary that $E_j$ is not fully $\cF$-normalized.  Now, we need to change the path. There is an $h \in H$ such that $E_{j'}=h E_j h^{-1}$ is fully $\cF$-normalized. 
$N_S(E_{j'})$ is a Sylow $p$-subgroup of $N_H(E_{j'})$, which lies in the stabilizer of the edge $C_1$ to $E_{j'}$. By taking the $h$-conjugate of the conclusion from the previous paragraph we see that the edge from $E_{j'}$ to $hEh^{-1}$ also satisfies the Sylow $p$-subgroup condition. 
Thus each of the edges in the path $C_1, E_{j'}, hEh^{-1}$ satisfies the Sylow $p$-subgroup condition, hence by Theorem \ref{thm:mainthm3}, we conclude that $\pi$ has a Sylow $p$-subgroup isomorphic to $S$.

For every $p$-centric subgroup $P \leq S$, let $E$ be the maximal central elementary abelian subgroup of $P$. Then $E$ is characteristic in $P$, so $N_G(P) \leq N_G(E)$. This means $N_G(P)$ stabilizes $E$. Hence, by Theorem \ref{thm:mainthm3}, we can conclude that  $\cF_S(\pi)=\cF_S(G)$.\end{proof}


\section{Signalizer functors for infinite group models}\label{sect:Signalizer}

Let $\cF$ be a saturated fusion system on a finite $p$-group $S$. The \emph{centric linking system} $\cL$ associated to $\cF$
is a category whose objects are the $\cF$-centric subgroups of $S$, together with a functor $\pi : \cL \to \cF^c$
and a monomorphism $\delta _P : P \to \Aut _{\cL} (P) $ for each $\cF$-centric subgroup $P \leq S$ satisfying certain properties 
(see \cite[Def 9.35]{Craven} for details).  The triple $(S, \cF, \cL)$ is called a \emph{$p$-local finite group} and its \emph{classifying space} 
is the space $|\cL | ^{\wedge} _p$.
When $\cF$ is realized by a finite group $G$, then $\cF$-centric subgroups of $S$ are $p$-centric 
in $G$. Hence for every $P\in \cF^c$, we have $C_G(P) =Z(P) \times C'_G(P)$ where $C'_G(P)$ has 
order coprime to $p$. In this case the morphisms of the category $\cL$ are given by 
$$\mor _{\cL} (P, Q) =\{ g C'_G(P) \, | \, g\in G, \, gPg^{-1} \leq Q \}= \{ g \in G \, | \, gPg^{-1} \leq Q \}/ C'_G(P).$$

For a discrete group $\pi$ with a Sylow $p$-subgroup $S$, the transporter category $\cT^c_S (\pi)$ is defined as the category whose objects are 
$\cF$-centric subgroups of $S$ and whose morphisms are given by
$$\mor _{\cT^c _S(\pi)} (P, Q) =\{ g\in \pi \, | \, gPg^{-1} \leq Q \}.$$ 

\begin{definition}\label{def:Signalizer} A \emph{signalizer functor} $\Theta$ on a discrete 
group $\pi$ is an assignment $P \to \Theta (P)$ for every $\cF$-centric $P \leq S$ such that $\Theta (P)$ is a complement 
of $Z(P)$ in $C_{\pi}(P)$ and such that if $gPg^{-1} \leq Q$ then $\Theta (Q) \leq g \Theta (P) g^{-1} $. 
\end{definition}

A signalizer functor $\Theta$ is a functor from the transporter category $\cT_S^c (\pi)$ to the category of groups.
Given a signalizer functor $\Theta$ on a discrete group $\pi$, we can define a 
quotient category $\cL_{\Theta}$ whose morphisms from $P$ to $Q$ are $$\mor _{\cL} (P, Q) = \mor_{\cT ^c_S(\pi) } (P, Q) / \Theta (P).$$ 
It is proved in \cite[Lemma 2.6]{AschbacherChermak} that the category $\cL_{\Theta}$ is a linking system for the fusion
system $\cF_S(\pi)$.  For infinite group models $\pi_R$ and $\pi_{LS}$ realizing fusion systems, we have the following theorem.

\begin{theorem}[Libman and Seeliger, \cite{LibmanSeeliger}]\label{thm:LibmanSeeliger2} 
Fix a $p$-local finite group $(S,\mathcal{F},\mathcal{L})$ and let $\pi$ be an 
infinite group model realizing $\cF$ obtained by either the Leary-Stancu model or the Robinsion model. 
Then there is a signalizer function $\Theta$ on $\pi$ such that 
$\cL$ is a quotient of the transporter system $\cT _S ^c (\pi)$.
\end{theorem} 
 
Theorem \ref{thm:LibmanSeeliger2} is proved as part of \cite[Thm 1.1]{LibmanSeeliger} using 
topological arguments, in particular, using the definition of linking system associated to a map $f: BS \to X$.  In the proof it is shown that for each $P\in \cF^c$,
there is a group extension 
$$ 1 \to \Theta (P) \to N_{\pi} (P) \maprt{\chi_P} \Aut _{\cL} (P) \to 1.$$
The group $\Aut _{\cL} (P)$ has $N_S(P)$ as a Sylow $p$-subgroup, and it is known that the induced fusion system
is isomorphic to the normalizer fusion system $N_{\cF} (P)$.  If $\pi$ is an infinite group model constructed by the
Leary-Stancu model or the Robinson model, then $\pi$ has a Sylow $p$-subgroup isomorphic to $S$, but in general this property
 is not inherited by subgroups. So, it is not clear that 
 the normalizer group $N_{\pi} (S)$ has a Sylow $p$-subgroup.
 
 \begin{question} Does the normalizer group $N_{\pi} (P)$ in general have a Sylow $p$-subgroup isomorphic to $N_S(P)$? Is it possible to see
 $N_{\pi} (P)$ as an infinite model for the fusion system $N_{\cF} (P)$ which is realizable by the finite group $\Aut _{\cL} (P)$ and view the above
 sequence as a sequence coming from a storing homomorphism of a realizable fusion system?
 \end{question} 
 
 \begin{example}\label{ex:Theta} Let $\cF$ be a fusion system of a finite group $G$ with Sylow $p$-subgroup $S$. Assume that $G$ is 
 $p$-minimal, and that $\pi$ is an infinite group realizing $\cF$ obtained by either the Leary-Stancu model or by the Robinson model. There is a storing homomorphism $\chi: \pi \to G$ whose kernel is a free group $F$. For each $P \in \cF^c$, reducing the storing homomorphism to $N_{\pi}(P)$, we obtain a short exact sequence 
 $$ 1 \to N_F(P) \to N_{\pi } (P) \to  \overline N_{\pi} (P)  \to 1,$$ where  $\overline N_G(P)= \chi (N_{\pi} (P) )$ is a subgroup of $N_G(P)$.  
By Theorem \ref{thm:LibmanSeeliger2}, the homomorphism $$\chi_P : N_{\pi} (P) \to \Aut _{\cL } (P)= N_G(P)/ C_G' (P)$$ is surjective, hence $\overline N_{\pi} (P) C'_G(P)=N_G(P)$. This gives that $\Theta$ fits into an extension of the form $$1 \to N_F(P) \to \Theta (P) \to \overline \Theta (P) \to 1$$ where $\overline \Theta (P) =\overline N_{\pi } (P) \cap C' _G(P)$. Note that $\overline \Theta (P)$ is a finite group whose order is coprime to $p$ and $N_F(P)\leq F$ is a free group.
 \end{example}
 
 One of the consequences of the calculation given in Example \ref{ex:Theta} is that mod-$p$ cohomology of the 
 group $\Theta (P)$ is zero at dimensions greater than $1$. It turns out that this holds more generally for any saturated 
 fusion system.
 
\begin{proposition}\label{pro:Homology} Let $\cF$ be a saturated fusion system on a finite $p$-group $S$. Suppose $\pi$ is an infinite group model
obtained by either the Leary-Stancu model or the Robinson. Let $P \to \Theta(P)$ be the signalizer functor
for $\pi$ such that $\cL$ is a quotient of the transporter system $\cT _S ^c (\pi)$. 
Then for every $P\in \cF ^c $ and for every $i \geq 2$, we have $H_i (\Theta (P); \bF_p)=0$. 
\end{proposition}
 
\begin{proof} Note that $\pi$ has a Sylow $p$-subgroup $S$ and $\cF=\cF_S(\pi )$. Let $P$ be an $\cF$-centric subgroup of $S$.
Since $\Theta (P)$ is a subgroup of the fundamental group of a graph of groups with finite vertex groups, it is itself a graph of groups with finite
vertex groups. The result follows from the cohomology sequence for a graph of groups once we show that all finite vertex groups of $\Theta (P)$ have order coprime to $p$. Let $Q \leq \Theta (P)$ be a finite $p$-group. Then $Q$ centralizes $P$, hence $PQ$ is a finite $p$-subgroup of $\pi$.
Let $g\in \pi$ be such that $g(PQ)g^{-1} \leq S$. Then $P'=gPg^{-1}$ is a subgroup of $S$ and $gQg^{-1} \leq S$ centralizes $P'$. Since $P$ is $\cF$-centric, and $P'$ is $\cF$-isomorphic to $P$, we have $C_S(P')=Z(P')$. This gives that $gQg^{-1}\leq Z(P')$, which implies $Q\leq Z(P)$. Since $\Theta (P)\cap Z(P)=1$, we get $Q=1$. We conclude that all finite subgroups of $\Theta(P)$ have order coprime to $p$. 
\end{proof}

The proposition we proved above has consequences for cohomology groups of the category $\cL$ with twisted coefficients. For a commutative ring $R$, an \emph{$R \cL$-module $M$} is defined as a contravariant functor from $\cL$ to the category of $R$-modules. Let $\cT:=\cT_S ^c (\pi)$ denote the transporter category for $\pi$ defined on $\cF$-centric subgroups of $S$. 
The quotient functor $q: \cT \to \cL$ gives an extension of small categories $$ 1 \to \{ \Theta (P) \} \to \cT \to \cL \to 1.$$ 
Theorem \ref{thm:thmB} is proved using a spectral sequence for extensions of small categories. We first introduce necessary definitions. We refer the reader to \cite[Appendix A]{OliverVentura} for details.

\begin{definition}[Def A.5, \cite{OliverVentura}] Let $\cC$ and $\cD$ be two small categories such that $\Ob (\cC)=\Ob(\cD)$. Let $\varphi: \cC \to \cD$ be a functor that is identity on objects and surjective on morphism sets. For each $x\in \Ob (\cC)$, set $$K(x) :=\ker \{ \varphi _{x,x} : \Aut _{\cC} (x) \to \Aut _{\cD} (x)\}.$$ We say $\varphi$ is \emph{source regular} if for every $x, y \in \Ob(\cC)$, the group $K(x)$ acts freely on $\mor _{\cC} (x,y) $ and $\varphi _{x,y} : \mor _{\cC} (x,y) \to \mor _{\cD} (x,y)$ is the orbit map of this action. 
\end{definition}

When $\varphi : \cC\to \cD$ is a source regular functor, we say  
$$1 \to \{ K(x) \} \to \cC \maprt{\varphi} \cD \to 1$$ is a source regular extension, 
and call the family $\{K(x) \}$ the kernel of the functor $\varphi$. For a small category $\cC$, an $R\cC$-module is defined as a functor $M : \cC ^{op} \to R$-mod. Note that the assignment $x \to H_i (K(x) ; R)$ defines an $R\cD$-module since $K(x)$ acts trivially on homology groups $H_i (K(x); R)$ when $R$ denotes a trivial $RK(x)$-module. Given a functor $\varphi : \cC\to \cD$ and an $R\cD$-module $M$, the $R\cC$-module defined as the composition $$\cC \maprt{\varphi} \cD \maprt{M} R{\rm -mod} $$
is denoted by $\varphi ^* M$.

\begin{proposition}\label{pro:Spectral}
Let $\varphi : \cC\to \cD$ be a source regular functor with kernel $\{ K(x) \}$. Then for every $R\cD$-module $M$, there is a spectral sequence
$$E_2 ^{i,j} =\ext ^j _{R \cD }  (H_i (K(-); R) , M) \Rightarrow H^{i+j} (\cC ; \varphi^* M).$$
\end{proposition}

\begin{proof} There is a standard resolution for calculating the cohomology groups of a small category $\cC$ with coefficients in an $R\cC$-module $M$.
For each $n$, let $C_n (-)$ be
defined as the $R\cC$-module $$C_n (-) =\bigoplus _{x_0 \to \cdots \to x_n } R \mor _{\cC}  (-, x_0 )$$
where the sum is over all sequences of $n$ morphisms $x_0 \to \cdots \to x_n$ in $\cC$ (see \cite[Lemma A.1]{OliverVentura} for details). The complex $C_* (-)$ gives a 
projective resolution of the constant functor $\underline R$ as an $R\cC$-module. The cohomology of $\cC$ with coefficients in $\varphi^* M$ is defined as the cohomology of the cochain complex $\hom _{R\cC} ( C_*  (-) ; \varphi^*M).$ It is easy to see that 
$$\hom _{R\cC} (\mor _{\cC} (- , x_0 ) , q^*M  )\cong M (x_0 ) \cong \hom _{R\cD } ( \mor _{\cD} (-, x_0 ), M).$$ 
Hence, we can write 
$$H^* (\cC ; \varphi ^* M ) =H^* ( \hom _{R\cC} ( C_*  (-) ; M))\cong H^* (\hom _{R\cD} ( C'_*  (-) ; M))$$
where $C'_* (-) $ is the complex of projective $R\cD$-modules such that 
$$C'_n (x) = \bigoplus _{x_0 \to \cdots \to x_n } R [\mor _{\cC}  (x, x_0 )/ K(x)]$$
for every $x\in \Ob (\cC)$. Alternatively we can consider $C_n'(-)$ as the chain complex obtained from $C_*(-)$
by applying the functor $(-)_K$ introduced in the proof of \cite[Lem 1.3]{BLO1}.

For the calculation of the cohomology group $H^* (\hom _{R\cD} ( C'_*  (-) ; M))$, there is an hyper-cohomology 
spectral sequence (see \cite[Prop 3.4.3]{Benson2}) with $E_2$-term
$$E_2 ^{i,j} =\ext ^j _{R \cD }  (H_i ( C'_* (- ); R) , M  ) \Rightarrow H^{i+j} (\cC ; \varphi ^* M).$$

Since $K (x) $ acts freely on $\mor _{\cC} (x, y)$ for every $x, y \in \Ob (\cC)$, we have 
$$H_i ( C'_* (x); R)  \cong H_i (K(x); R)$$ for every $x\in \Ob (\cC)$, and for every $i$. 
This completes the proof.
\end{proof}

\begin{proof}[Proof of Theorem \ref{thm:thmC}] 
Let $R=\bF_p$. The quotient functor $q: \cT \to \cL$ is a source regular functor with kernel $\{ \Theta (P)\}$. 
Hence by Proposition
\ref{pro:Spectral}, for every $R\cL$-module $M$, there is a spectral sequence
$$E_2 ^{i,j} =\ext ^j _{R \cL }  (H_i (\Theta (-) ; R) , M) \Rightarrow H^{i+j} (\cT ; q^* M).$$
By Proposition \ref{pro:Homology}, $H_i (\Theta (P); R)=0$ for every $i \geq 2$, hence the $E_2$-term of the spectral sequence 
has only two nonzero 
horizontal lines at $i=0,1$ with a single differential $$d_2 : \ext ^j _{R \cL }  (H_1 ( \Theta (-) ; R) , M  ) \to 
\ext ^{j+2} _{R \cL }  (R , M  )$$ for $j \geq 0$. The long exact sequence given in the statement of the theorem follows
from the fact that $E_2$-page is a two line spectral sequence. 
\end{proof}


\end{document}